\documentclass{svproc}
\usepackage{graphicx}%
\usepackage{multirow}%
\usepackage{amsmath,amssymb,amsfonts}%
\usepackage{mathrsfs}%
\usepackage[title]{appendix}%
\usepackage{xcolor}%
\usepackage{textcomp}%
\usepackage{manyfoot}%
\usepackage{booktabs}%
\usepackage{algorithm}%
\usepackage{algorithmicx}%
\usepackage{algpseudocode}%
\usepackage{listings}%
\usepackage{url}
\usepackage{bbm}
\usepackage[capitalise]{cleveref}

\DeclareMathOperator{\E}{\mathbb{E}}

\newcommand{\1}{\mathbbm{1}}
\newcommand{\K}{\Omega}
\newcommand{\B}{\mathcal{B}}

\renewcommand{\phi}{\varphi}
\newcommand{\T}{\mathrm{T}}
\newcommand{\rr}{\mathbb{R}}

\newcommand{\set}[1]{\{ #1\}}
\newcommand{\ol}[1]{\overline{#1}}
\newcommand{\dd}{\,\mathrm{d}}
\newcommand{\abs}[1]{| #1 |}
\begin{document}
\mainmatter              
\title{Explicit Bounds on the Entropy of Piecewise H\"{o}lder Graphon Models}
\titlerunning{Explicit Bounds on the Entropy of Graphon Models}  
%
\author{Connor Loehde-Woolard \and Fran\c{c}ois G. Meyer}
\authorrunning{Connor Loehde-Woolard \and Fran\c{c}ois G. Meyer} 

\institute{University of Colorado Boulder\\
\email{connor.loehde-woolard@colorado.edu}}

\maketitle              

\begin{abstract}
We study the entropy of random graphs generated by piecewise H\"{o}lder continuous graphons. We first present a result on the rate of convergence of the normalized entropy as the size of the graph grows. The core ideas of the proof are described, with the detailed proof provided in the appendix. From this result, we then derive quantitative bounds on the entropy for the stochastic block model and random geometric graph model. These bounds provide explicit formulae rather than asymptotic statements which have been found previously.

\keywords{network entropy, network complexity, graphon,
graph limit, network compression, information theory of networks}
\end{abstract}

\section{Introduction}\label{sec1}

The problem of compressing graph-valued data has attracted significant interest for  addressing practical limitations to the storage and use of such data \cite{gao:2025,xu:2025}. As a result, it has become valuable to be able to quantify the information content, or entropy, of different random graph models as the entropy gives limits on the lossless compression of data from a given source \cite{cover:2006}. Characterizations have been found for the entropy of commonly-used models including Erd\H{o}s-R\'{e}nyi \cite{wafula:2023}, the stochastic block model (SBM) \cite{abbe:2016}, and the hard random geometric graph model (RGG) \cite{vippathalla:2026}. In the case of the latter two, these results are limited to describing the asymptotic behavior of the entropy as the number of nodes increases, and do not provide explicit quantitative formulae. In this paper, we address this problem by extending recent work by Baker et al. \cite{baker:2026} which studies the limits of compression for the soft random geometric graph model (SRGG). The authors obtain bounds on the convergence of the SRGG entropy in the limit of large graph size. This is accomplished by viewing the connection function in the SRGG as a graphon. A graphon is a function of two variables which arises as the limiting mathematical object of certain sequences of graphs, and any graphon $W$ gives rise to a random graph model. In their paper, the authors derive results for the case of a graphon which is globally H\"{o}lder continuous on its domain. The SBM and RGG can also be viewed as models generated by certain graphons, but these graphons are only piecewise H\"{o}lder continuous. We modify the arguments in \cite{baker:2026} to obtain results for such models. In particular, we prove that the normalized entropy of a $W$-random graph, with a graphon that is piecewise H\"{o}lder continuous on an appropriate domain, converges with a rate that is $O(\frac{\log n}{n})$ in the limit of large graph size. From this result, we also obtain explicit bounds on the entropy of the SBM and RGG. 

The remainder of the paper is as follows. In \cref{sec2} we provide the necessary background. In \cref{sec3} we present the main result on the entropy of graphs generated from a piecewise H\"{o}lder graphon and some corollaries. Concluding remarks are given in \cref{sec5}. The detailed proof of the main result is given in Appendix \ref{app1}.

\section{Background}\label{sec2}

In this paper, by the term graph we mean a simple, undirected, unweighted graph $G = (V,E)$ with vertex set $V = \{1,\dots,n\}$ and edge set $E$ consisting of unordered pairs of elements in $V$. 

\subsection{Random Graphs and Graphons}
Let $\mathcal{G}_n$ be the set of graphs on $n$ vertices. A particular random graph model corresponds to a probability distribution $f$ on $\mathcal{G}_n$. Many are also generated via some random process, such as the SBM. Given a number of nodes $n$, a vector $v\in\rr^k$ satisfying $v_i\geq 0$ for all $i$ and $\sum_{i=1}^k v_i = 1$, and a symmetric matrix $\Phi\in [0,1]^{k\times k}$, the \emph{stochastic block model} $\mathrm{SBM}(n,v,\Phi)$ is defined by sampling $X_1,\dots,X_n$ from $\set{1,\dots,k}$ independently with probabilities given by $v$, and connecting nodes $i$ and $j$ with probability $\Phi_{X_i,X_j}$ independently of all other edges. A similar process defines the random geometric graph model for a given number of nodes $n$, a subset $\mathcal{X}\subset\mathbb{R}^d$ of Euclidean space with finite diameter $D = \sup_{x,y\in \mathcal{X}}\|x-y\|_2$, and a function $p:[0,D]\to[0,1]$. Here, points $X_1,\dots, X_n$ are sampled uniformly from $\mathcal{X}$ and nodes $i$ and $j$ independently share an edge with probability $p(\|X_i - X_j\|)$. The \emph{hard random geometric graph model} $\mathrm{RGG}(n,\mathcal{X},\tau)$ is obtained when $p$ is of the form
\[p(x) = \begin{cases}1 & \text{if}~x\leq\tau,\\ 0 & \text{if}~x>\tau,\end{cases}\]
for a given threshold $\tau$. The SRGG extends this by allowing for more general connection functions. In the case that the latent space is $\mathcal{X}=[0,1]^d$, we will denote the hard RGG by $\mathrm{RGG}(n,d,\tau)$.

The related underlying processes in the SBM and geometric graph models can be united through the notion of a graphon. Given a probability space $(\Omega, \mathcal{F},\mu)$, a \emph{graphon} is a symmetric measurable function $W:\Omega^2\to[0,1]$. These functions first arose as a tool to study the behavior of sequences of dense graphs in the limit as the number of nodes increases \cite{lovasz:2012}. Given a graphon $W$ on a space $\Omega$, the $W$-random graph model on a number of nodes $n$ is defined by sampling points $X_1,\dots,X_n$ uniformly from $\Omega$ and including the edge between nodes $i$ and $j$ independently with probability $W(X_i,X_j)$. It is clear that this generalizes the geometric graph models, where the graphon is given by $W(x,y) = p(\|x-y\|)$. Similarly, we can express the model $\mathrm{SBM}(n,v,\Phi)$ as a graphon model by partitioning the unit interval as
\[[0,1] =[a_1,b_1)\cup [a_2,b_2)\cup\cdots\cup [a_k,b_k]\]
where $a_1=0$ and $b_k=1$, and $b_i-a_i=v_i$ for $i=1,\dots,k$; we write $I_i$ to denote $[a_i,b_i)$. We then define a graphon $\phi:[0,1]^2\to[0,1]$ by
\begin{equation}\label{sbm_graphon}\phi = \sum_{i,j}\Phi_{ij}\1_{I_i\times I_j}.\end{equation}

Common to the SBM and the RGG is that the graphons defining them are not necessarily continuous over the domain. In order to derive the desired quantitative results, it is necessary for us to define the extent of smoothness we require for the graphons. A graphon $W:\K^2\to[0,1]$ is \emph{piecewise H\"{o}lder continuous} if $W$ is of the form
\begin{equation}W(x,y) = \sum_{t=1}^T f_t(x,y) \1_{A_t}(x,y)\end{equation}
where for each $t$, $f_t$ is H\"{o}lder continuous on $A_t$ with constant $L_t$ and exponent $\alpha_t$, and
\begin{equation}\bigcup_{t=1}^T A_t =\K^2\end{equation} 
for sets $A_t$ which are mutually disjoint. We will also refer to such a graphon as simply \emph{piecewise H\"{o}lder}. We observe that the graphons for the two models mentioned above are piecewise H\"{o}lder, and further for the SBM and hard geometric model $\mathrm{RGG}(n,1,\tau)$ the sets $A_t$ partitioning the domain fall into a special class, which we next define. 

By a \emph{relative polytope} we mean a bounded set in $\mathbb{R}^d$ with non-empty interior, which is the non-empty intersection of a finite number of half-spaces. These half spaces may be open or closed. If $P$ is a relative polytope, the closure $\ol{P}$ is a convex polytope in the traditional sense.

\subsection{Entropy and Graphs}\label{entropyback}
For a discrete random variable $Y$ supported on a set $\mathcal{Y}$ with distribution function $P_Y$, the \emph{entropy} of $Y$ is defined to be
\[H(Y) = -\sum_{y\in\mathcal{Y}}P_Y(y)\log\big(P_Y(y)\big),\]
with the convention that $0\cdot\log(0)=0$. We write $\log$ to denote the base-2 logarithm. For the special case of a Bernoulli random variable $B\sim\mathrm{Bernoulli}(p)$ the entropy is given by $h(p)$, where $h$ is the \emph{binary entropy function} defined by $h(x) = -x\log(x)-(1-x)\log(1-x)$. The entropy of a given source is of fundamental importance in information theory, and particularly in the study of compression where the entropy gives a hard lower bound on the minimum expected length of a prefix-free code \cite[Th. 5.4.1]{cover:2006}.

 The entropy of a random graph $G$ on $n$ nodes with distribution $P_G$ is similarly defined to be
\[H(G) = -\sum_{g\in\mathcal{G}_n}P_G(g)\log\big(P_G(g)\big).\]
Entropy has been studied for many standard random graph models. For the model $\mathrm{SBM}(n,v,\Phi)$, the entropy has been characterized in \cite{abbe:2016} up to an asymptotic term.

\begin{proposition}[Abbe, \cite{abbe:2016}]\label{abbe}
Let $G\sim\mathrm{SBM}(n,v,\Phi)$. Then
\[H(G) = \binom{n}{2}v^\T H(\Phi) v+O(n),\]
where $[H(\Phi)]_{ij} = h(\Phi_{ij})$.
\end{proposition}

In a similar manner, recent work in \cite{vippathalla:2026} gives an upper bound on the entropy for the hard RGG with latent space $[0,1]^d$ to leading order. 

\begin{proposition}[Vippathalla et al., \cite{vippathalla:2026}]\label{rgg}
Let $G\sim\mathrm{RGG}(n,d,\tau)$. Then
\[H(G)\leq\begin{cases}dn\log n+o(n\log n) & \text{if}~0<\tau\leq\sqrt{d}/2,\\ [1-\beta(\tau)]dn\log n+o(n\log n) & \text{if}~\sqrt{d}/2\leq \tau<\sqrt{d},\end{cases}\]
where $\beta(\tau)$ is the volume of the ball $B\big((1/2,\dots,1/2); r-\sqrt{d}/2\big)\cap[0,1]^d$.
\end{proposition}

The following result, proved as Theorem D.5 in \cite{janson:2013}, determines the limiting behavior of $W$-random graph entropy as the number of nodes grows. 

\begin{proposition}[Janson, \cite{janson:2013}]\label{jan_prop}
Let $W$ be a graphon defined on a set $\Omega$ with probability measure $\mu$. Then 
\begin{equation}\label{lowbnd}
H\big(G(n,W)\big)\geq\binom{n}{2} \iint_{\Omega^2}h\big(W(x,y)\big)\dd\mu(x)\dd\mu(y),
\end{equation}
where $H\big(G(n,W)\big)$ is the entropy of $G(n,W)$. Further,
\begin{equation}
\lim_{n\to\infty}\frac{H\big(G(n,W)\big)}{\binom n2} = \iint_{\Omega^2}h\big(W(x,y)\big)\dd\mu(x)\dd\mu(y).
\end{equation}
\end{proposition}

\section{Convergence of Entropy for Piecewise H\"{o}lder Graphons}\label{sec3}

Consider a set $\K\subset\rr^d$ which is bounded with unit Lebesgue volume and finite surface area. Then $\K$ can be contained in a bounding $d$-cube $\B$ with side length $\ell$. We can define a graphon $W$ on $\K$ and attempt to determine the entropy of a $W$-random graph on $n$ nodes $G(n,W)$. Conditioned on the points $X_1,\dots, X_n$ in $\Omega$, the edge indicator variables associated to $G(n,W)$ are independent and the conditional entropy can be computed directly. Thus a significant challenge in determining the total entropy in this setting is dealing with the unknown node positions. Based on the technique of proof given in \cite{janson:2013} for \cref{jan_prop}, Baker et al. \cite{baker:2026} prove the following result. 
\begin{proposition}[Baker et al., \cite{baker:2026}]\label{baker_prop}
For $n\geq 2$, let $G$ be an SRGG on $n$ nodes with a H\"{o}lder continuous connection function $p$. Suppose $p$ has H\"{o}lder exponent $\alpha$ and H\"{o}lder constant $L$. Define
\[\Delta_n = \frac{1}{\binom{n}{2}}H(G)-\E[h(p(R))].\]
Then
\[0\leq\Delta_n\leq\frac{dn\log(m)}{\binom{n}{2}}+a_1\frac{\log(m)}{m^\alpha}+a_2\frac{1}{m^\alpha}+a_3\frac{1}{m}+a_4\frac{1}{m^2},\]
where $m = Cn^\gamma$ for some $\gamma>0$ and some constant $C \geq L^{1/\alpha}\ell\sqrt{2d}$. The constants $a_i$, $i=1,2,3,4$ are given by
\begin{align*}
a_1 &= \frac{\alpha L}{\ln 2}(\ell\sqrt{2d})^\alpha, & a_2 &= L(\ell\sqrt{2d})^\alpha \big(1-\log\big(L(\ell\sqrt{2d})^\alpha\big)\big),\\
a_3&= 2|\partial\K|\ell, & a_4 &= |\partial\K|^2\ell^2,
\end{align*}
where $|\partial\K|$ is the $(d-1)$-dimensional surface area of $\K$ . 
\end{proposition}

In the case of random graph models we are most interested in, the latent space $\Omega$ can be taken as the unit interval. For this reason and in the interest of simplicity, for the remainder of the paper we assume $\Omega$ is a product of intervals. The assumption of H\"{o}lder continuity is important in the proof of \cref{baker_prop} for obtaining a bound with sufficient decay. As noted above, this assumption fails globally for some important and common models, while piecewise H\"{o}lder continuity holds. This raises the question of how discontinuities present in the graphon affect the bound, and if the same rate of decay is achieved.
As we also note above, the graphons for the SBM and $\mathrm{RGG}(n,1,\tau)$ models take the form
\begin{equation}W(x,y) = \sum_{t=1}^T f_t(x,y) \1_{A_t}(x,y),\end{equation}
where for each $t=1,\dots,T$, $A_t = \K^2\cap R_t$ where $R_t$ is a relative polytope and $A_t$ has non-empty interior. We refer to such a graphon as \emph{piecewise H\"{o}lder with polytope domain}. We can assume that $R_t\subset\B$, as if not we can take the intersection $R_t'=R_t\cap\B$ to obtain a new relative polytope which still satisfies $A_t = \K^2\cap R_t'$. The boundary of these sets is sufficiently regular that the discontiuities in the graphon do not significantly alter the bound from \cref{baker_prop}, as we state below.

\begin{proposition}\label{prop1}
Let $W$ be piecewise H\"{o}lder with polytope domain, and write
\[W(x,y) = \sum_{t=1}^T f_t(x,y) \1_{A_t}(x,y),\]
where for each $t$, $f_t$ is H\"{o}lder continuous with constant $L_t$ and exponent $\alpha_t$. Define $L=\max_{t\in\set{1,\dots,T}}L_t$ and $\alpha = \min_{t\in\set{1,\dots,T}}\alpha_t$. Further, let 
\begin{equation}\rho = \max_{t,t'\in\set{1,\dots,T}}\sup_{\substack{(x,y)\in A_t\\ (x',y')\in A_{t'}}}|f_t(x,y)-f_{t'}(x',y')|.\end{equation}
Define
\[\Delta_n = \frac{1}{\binom n2}H(G(n,W)) - \iint_{\Omega^2}h\big(W(x,y)\big)\dd x\dd y.\]
Then
\begin{equation}0\leq\Delta_n\leq\frac{dn\log m}{\binom n2} +a_1 \frac{\log m}{m^\alpha}+a_2 \frac{1}{m^\alpha}+a_3\frac{1}{m}+a_4\frac{1}{m^2}\end{equation}
where $m = Cn^\gamma$ for $\gamma>0$ with $C\geq L^{1/\alpha}\ell\sqrt{2d}$ and the constants $a_i$ are given by 
\begin{align*}
a_1 &= \alpha L (\ell\sqrt{2d})^{\alpha}\ell^{2d}, & a_2 &=L(\ell\sqrt{2d})^{\alpha}(1/\ln 2 - \log(L(\ell\sqrt{2d})^{\alpha}))\ell^{2d},\\
a_3 &= 4d(T-1)\ell^{2d}\theta+2|\partial\K|\ell, & a_4 &= |\partial\K|^2\ell^2,
\end{align*}
where $\theta = -\rho\log\rho+\rho/\ln 2$. In the special case that $\Omega = [0,1]$, we can take $a_3 = 4d(T-1)\ell^{2}\theta$ and $a_4 = 0$.
\end{proposition}

The proof is a modification of the proof of \cref{baker_prop}; we omit the full proof due to page count constraints. The core idea is to condition on a discretization of the latent space rather than the exact positions associated to the nodes. The entropy associated to the discretization can be bounded in a standard manner, and the conditional entropy of $G(n,W)$ is bounded by a term based on a block-constant approximation of $W$. The difference $\Delta_n$ is then bounded by considering the variation of $W$ over each region determined by the resolution of the discretization. If $W$ is H\"{o}lder continuous on the region, this variation is bounded in terms of the size of the region. Otherwise, the variation is bounded by $\rho$. As the resolution of the discretization increases, the fraction of regions over which $W$ is not H\"{o}lder continuous vanishes, and the discontinuity does not introduce a factor which dominates the leading order term from \cref{baker_prop}. For the full proof, see Appendix \ref{app1}.

As with the result from \cite{baker:2026}, taking $\gamma=1/\alpha$ in \cref{prop1} immediately gives that $\Delta_n$ converges to 0 at a rate that is $O(\frac{\log n}{n})$. 

\subsection{Bounds on the Entropy of the Stochastic Block Model}\label{sbm}
The result in \cref{prop1} gives us not only a rate of convergence, but also an upper bound on the entropy of a $W$-random graph with a given number of nodes. It is then of interest to ask how this bound compares with the results stated in \cref{entropyback}.  From \cref{abbe}, we have that the entropy of a stochastic block model graph $G\sim\mathrm{SBM}(n,v,\Phi)$ is given by
\[H(G) = \binom{n}{2}v^\T H(\Phi) v+O(n).\]
Defining a graphon $\phi$ as in \cref{sbm_graphon}, the following remark is elementary to verify.
\begin{remark}
Given a model $\mathrm{SBM}(n,v,\Phi)$ and equivalent graphon $\phi$, we have
\[v^\T H(\Phi)v = \iint_{[0,1]^2}h(\phi(x,y))\dd x\dd y.\]
\end{remark}

We can then rewrite the result from \cref{abbe} to see
\[H(G) = \binom{n}{2}\iint_{[0,1]^2}h(\phi(x,y))\dd x\dd y+O(n),\]
which together with \cref{jan_prop} improves on the rate of convergence from \cref{prop1} in this special case, as it implies
\[\frac{1}{\binom n2}H(G(n,W)) -\iint_{[0,1]^2}h(\phi(x,y))\dd x\dd y = O\Big(\frac{1}{n}\Big).\]

Notably, this has a lower leading order than the bound from \cref{prop1}, but gives only an asymptotic characterization. As the graphon $\phi$ is piecewise constant on $[0,1]^2$, we can take $L=\alpha=\ell=d=1$ with $T=k^2$ in the statement of \cref{prop1}, from which it follows we can choose $C=2$. In the notation from \cref{abbe}, our result applied to the $\mathrm{SBM}(n,v,\Phi)$ case yields
\begin{align}\label{sbm_result}\binom{n}{2}v^\T H(\Phi)v\leq H(G)&\leq\binom{n}{2}v^\T H(\Phi)v+\Big(\frac{4+\sqrt{2}}{4}\Big)n\log(2n) \nonumber \\ 
&+\Big(\frac{2-\ln 2}{8\ln 2}\sqrt{2}+(k^2-1)\theta\Big)(n-1)-\frac{\sqrt{2}}{4}\log(2n)\end{align}
after simplifying constants.

\subsection{Bounds on the Entropy of the Random Geometric Graph Model}

The hard random geometric graph model $\mathrm{RGG}(n,1,\tau)$ corresponds to a graphon that is 0-1 valued, so the integral term in \cref{prop1} evaluates to 
\[\iint_{\Omega^2}h\big(W(x,y)\big)\dd x\dd y=0,\] 
since $h(0)=h(1)=0$. As before, since the graphon is piecewise constant on the square we can take $L=\alpha=\ell=d=1$, with $T=3$ and $C=2$. Applying \cref{prop1} to $G\sim\mathrm{RGG}(n,1,\tau)$ yields
\[H(G)\leq n\log(2n)+\frac{\sqrt{2}}{4}(n-1)\log(2n)+\Big(\frac{2-\ln 2}{8\ln2}\sqrt{2}+2\theta\Big)(n-1)\]
irrespective of $\tau$. This gives a bound with the same leading order as the asymptotic statement from \cref{rgg}.

\section{Numerical Simulation}\label{sec4}
We compare the bound given in \cref{sbm} to a numerical estimate of the entropy of a stochastic block model $\mathrm{SBM}(n,v,\Phi)$ for a number of nodes $n$ from 6 to 20. For each $n$, the model has two communities and the parameters $v$ and $\Phi$ used are
\begin{align*}
v &= \begin{pmatrix} 1/2 & 1/2\end{pmatrix}^\mathrm{T},\\
\Phi &=\begin{pmatrix}0.7&0.3\\ 0.3 & 0.7\end{pmatrix}.
\end{align*}

To estimate the entropy, we recall the identity
\[H(X) = -\E[\log(P(X))],\] 
where $X$ is a discrete random variable with probability mass function $P$ \cite{cover:2006}. A sample $g_1,\dots,g_N$ of size $N=50$ is drawn from the model, and we compute the probability $P_G(g_i)$. This computation is the primary bottleneck in this process as it has a computational cost that grows exponentially with the number of nodes. The estimate of the entropy, $\hat{H}$, is then computed as
\[\hat{H} = -\frac{1}{N}\sum_{i=1}^N \log(P_G(g_i)).\]
In \cref{fig1} we plot the computed entropy for each $n$ together with the bound given by \cref{sbm_result}. Error bars on the estimated entropy show 2 standard deviations of the numerical data. The number of nodes for which the computation could be performed is lower than would likely be common in practice, but a trend can be observed in the gap between the bound and the estimated entropy. 

\begin{figure}
\centering
\includegraphics[width=0.8\textwidth]{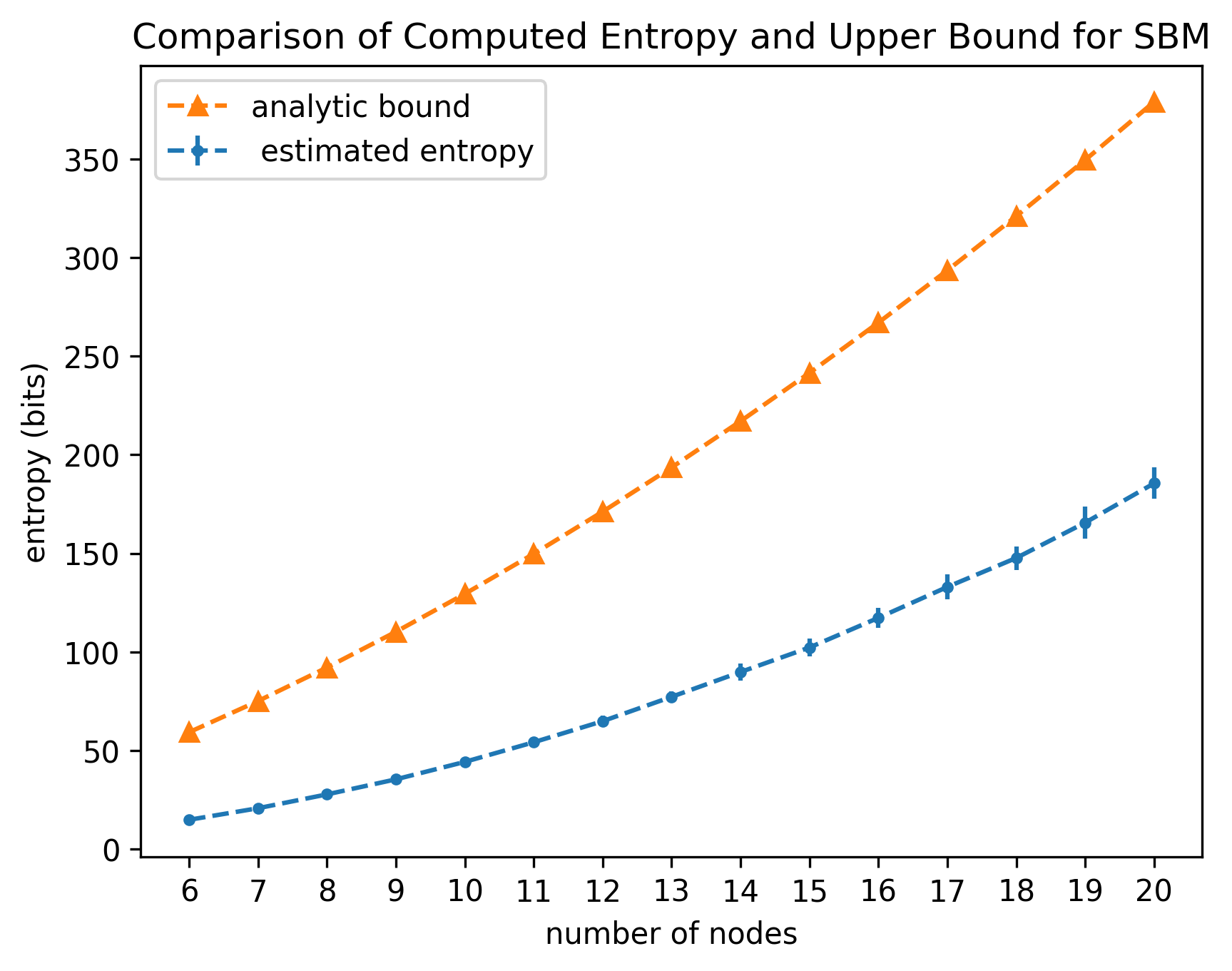}
\caption{Plot of estimated entropy and the upper bound given in \cref{sbm_result} for varying number of nodes. All estimates are for models with two communities and fixed parameters. Error bars on the estimated entropy show 2 standard deviations of the numerical data.}
\label{fig1}
\end{figure}

\section{Concluding Discussion}\label{sec5}

In this paper we prove a result on the rate of convergence of normalized $W$-random graph entropy under the assumption of piecewise H\"{o}lder continuity of the graphon. We also derive bounds on the entropy for two standard random graph models where only asymptotic characterizations had been given prior. These bounds may have a higher order of growth in terms of the number of nodes in the graph than previous results, but give an explicit formula for estimating the entropy. It is of note that the previous result of Abbe \cite{abbe:2016} implies a better rate of convergence than that which follows from \cref{prop1}. Future research may give a stronger convergence result than ours, though we suspect a different technique of proof would be necessary. It is also of interest to characterize the problem of lossy compression in the more general graphon setting. We note that the results on lossy compression from Baker et al. extend immediately to the framework we consider in this paper. A more practical direction of interest would be the modeling of real-world networks as graphon models of the type we consider in this paper.

\subsubsection{Acknowledgements}
This work utilized the Alpine High-Performance Computing resource at the University of Colorado Boulder. Alpine is jointly funded by the University of Colorado Boulder, the University of Colorado Anschutz, Colorado State University, and the National Science Foundation (award 2201538). 

\bibliographystyle{spmpsci} 
\bibliography{CNA_refs} 

\appendix
\section{Proof of \cref{prop1}}\label{app1}

The proof of \cref{prop1} follows the same approach as that of Theorem 3 in \cite{baker:2026}. The primary idea of the argument is to condition on a discretization of the latent space, and bound the contribution to the difference $\Delta_n$ from each element in the discretization. The H\"{o}lder assumption makes it convenient to bound this contribution for elements across which the graphon is continuous. There is only a trivial bound for elements which the graphon is not continuous over, but the polytope domain assumption ensures that the fraction of these elements will tend to zero as the discretization is made finer.

Consider covering $\K$ with $m^d$ disjoint $d$-cubes of side length $\ell/m$, which we will refer to as cells. Note that this also discretizes $\K^2$ into $m^{2d}$ disjoint product cells. There are two possibilities: either $\K$ can be neatly partitioned as a union of some collection of the cells, or it cannot. Consider the first case, in which no cells overlap the boundary of $\K$. Number the cells in some way, and define a random variable $M_i$ by $M_i = j$ if $X_i$ is in cell $c_j$. Define the vector $M = (M_1,\dots,M_n)$. Bounding $H(G(n,W))$ by the joint entropy, we have
\begin{equation}
H(G(n,W))\leq H(G(n,W)|M)+H(M).\end{equation}
Notice that $M$ has an alphabet of size $m^{dn}$ and so
\begin{equation}H(M)\leq dn\log(m).\end{equation}
We can bound the other term as
\begin{equation}\label{upbnd}
H(G(n,W)|M)\leq\sum_{j=1}^n\sum_{i=1}^{j-1} H(E_{ij}|M_i,M_j).
\end{equation}
To simplify this, we define $\ol{W}:\K^2\to [0,1]$ by
\begin{equation}
\ol{W}(x,y) = \E[W(X_1,X_2)|M_1 = i, M_2=j],
\end{equation}
for $x\in c_i, y\in c_j$. 

\begin{lemma}\label{olw}
\[H(E_{ij}|M_i,M_j) = \iint_{\K^2}h(\ol{W}(x,y))\dd x\dd y.\]
\end{lemma}

\begin{proof}
It is clear from the definition that
\begin{equation}P(E_{ij}=1|M_i=k, M_j=l) = \ol{W}(x_1,x_2),\end{equation}
for any $x_1\in c_k, x_2\in c_l$. Hence for any $x_1\in c_k, x_2\in c_l$
\begin{equation}
H(E_{ij}|M_i = k, M_j = l) = h(\ol{W}(x_1,x_2)),
\end{equation}
and so
\[
H(E_{ij}|M_i=k,M_j=l) = \Big(\frac{m}{\ell}\Big)^{2d}\int_{c_k}\int_{c_l}h(\ol{W}(x,y))\dd x\dd y.
\]
As the $X_i$ are i.i.d. uniformly on $\K$, $P(M_i = k, M_j=l) = \Big(\frac{\ell}{m}\Big)^{2d}$, and so
\begin{align}
H(E_{ij}|M_i,M_j) &=\sum_{k=1}^{m^d}\sum_{l=1}^{m^d}\int_{c_k}\int_{c_l}h(\ol{W}(x,y))\dd x\dd y\\
&= \iint_{\K^2}h(\ol{W}(x,y))\dd x\dd y.
\end{align}
\end{proof}

By \cref{olw} and \cref{upbnd} we then have
\begin{equation}
H(G(n,W)|M)\leq\binom{n}{2}\iint_{\K^2}h(\ol{W}(x,y))\dd x\dd y,
\end{equation}
which yields
\begin{equation*}\binom{n}{2}\iint\limits_{\K^2}h(W(x,y))\dd x\dd y\leq H(G(n,W))\leq dn\log(m)+\binom{n}{2}\iint\limits_{\K^2}h(\ol{W}(x,y))\dd x\dd y,\end{equation*}
when combined with \cref{lowbnd}.

Rearranging, we see we can bound the difference $\Delta_n$ by bounding
\begin{equation}\label{int}
\iint_{\K^2} h(\ol{W}(x,y))-h(W(x,y))\dd x\dd y.
\end{equation}
By Lemma 3 from \cite{baker:2026}, it is sufficient to bound
\begin{equation}\iint_{\K^2}h(|\ol{W}(x,y)-W(x,y)|)\dd x\dd y.\end{equation}
Let $C_1$ be the collection of product cells $B\subset\K^2$ such that $B$ is contained in one of the $A_t$, and let $C_2$ be the collection of product cells which are not. Note that $|C_1\cup C_2|=|C_1|+|C_2| \leq m^{2d}$. We would like to have a more precise estimate of the size of these two sets, for which we have the following lemmata.

\begin{lemma}[Lassak \cite{lassak:1988}]
Let $P\subset\rr^n$ be an $n$-dimensional parallelotope dissected into $\prod_{i=1}^n (p_i+1)$ smaller closed parallelotopes by $p_1+\dots+p_n$ hyperplanes, where $p_i$ is the number of hyperplanes parallel to the $i$-th pair of parallel $(n-1)$-dimensional faces of $P$. The boundary of any convex set $A\subset P$ can be covered with at most $\prod_{i=1}^n (p_i+1) - \prod_{i=1}^n (p_i-1)$ of the smaller parallelotopes.
\end{lemma}


\begin{lemma}
Let $X\subset\rr^d$ be a $d$-cube of side length $\ell$. Divide $X$ into $m^d$ disjoint $d$-cubes of side length $\ell/m$, which we refer to as cells. If $A\subset X$ is a relative polytope, let $N$ be the number of cells which intersect the boundary of $A$. Then 
\begin{equation}\label{shapelem}
 N\leq 2d m^{d-1}
\end{equation}
\end{lemma}
\begin{proof}
The case $m=1$ is trivial, so take $m\geq 2$. In the statement of \cite{lassak:1988}, take $p_1=\dots=p_n=m-1$ as we have a total of $m^d$ cells. Then we have
\[N\leq m^d - (m-2)^d.\]
Applying the mean value theorem to the function $x\mapsto x^d$ and observing that $x^{d-1}$ is non-decreasing on the interval $[m-2,m]$ then yields 
\[m^d - (m-2)^d\leq \big(m - (m-2)\big)d m^{d-1}.\]
The claim follows.
\end{proof}

By \cref{shapelem} we have that the number of cells intersecting the boundary of one set $A_t$ is at most $4d m^{2d-1}$. Each cell in $C_2$ is shared by at least 2 of the $T$ sets $A_t$, so after totalling the contributions from $T-1$ of the sets all have been accounted for. Then
\begin{equation}\label{c2bnd}
|C_2|\leq 4d(T-1)m^{2d-1}.
\end{equation}

Then
\begin{align*}\iint_{\K^2}h(|\ol{W}(x,y)-W(x,y)|)\dd x\dd y = &\sum_{B_1\in C_1}\iint_{B_1}h(|\ol{W}(x,y)-W(x,y)|)\dd x\dd y\\ &+\sum_{B_2\in C_2}\iint_{B_2}h(|\ol{W}(x,y)-W(x,y)|)\dd x\dd y.\end{align*}
For $B_1\in C_1$, since $W$ is H\"{o}lder continuous over $B_1$ we have
\begin{equation}|\ol{W}(x,y) - W(x,y)|\leq\sup_{(x,y),(x',y')\in B_1}\abs{W(x,y)-W(x',y')}\leq L\Big(\frac{\ell\sqrt{2d}}{m}\Big)^\alpha\end{equation}
where 
\begin{equation}L= \max_{t\in\set{1,\dots,T}}L_t,\quad \alpha = \min_{t\in\set{1,\dots,T}}\alpha_t.\end{equation}
This follows as the greatest distance between two points in $B_1$ is along a diagonal. To bound the integrand, we utilize the following fact.
\begin{lemma}[Baker \cite{baker:2026}]\label{ent-exp}
\begin{equation}\label{series}h(x) = -x\log x + \frac{x}{\ln 2}-\frac{1}{\ln 2}\sum_{k=2}^\infty \frac{x^k}{k(k-1)}.\end{equation}
\end{lemma}
From \cref{ent-exp}, we have $h(p)\leq -p\log p + p/\ln 2$; this upper bound is a non-decreasing function of $p$ on $[0,1]$. Then, provided $L(\ell\sqrt{2d}/m)^\alpha\leq 1$,
\begin{align*}
h(|\ol{W}(x,y)-W(x,y)|)&\leq -L\Big(\frac{\ell\sqrt{2d}}{m}\Big)^\alpha\log L\Big(\frac{\ell\sqrt{2d}}{m}\Big)^\alpha+\frac{L}{\ln 2}\Big(\frac{
\ell\sqrt{2d}}{m}\Big)^\alpha\\
&\leq\alpha L\Big(\frac{
\ell\sqrt{2d}}{m}\Big)^\alpha\log(m)+L\Big(\frac{\ell\sqrt{2d}}{m}\Big)^\alpha\Big(\frac{1}{\ln 2}-\log(L(\ell\sqrt{2d})^{\alpha})\Big)\\
&= a'_1 \frac{\log m}{m^\alpha}+a'_2 \frac{1}{m^\alpha}
\end{align*}
for constants $a'_1,a'_2$. Hence
\begin{align*}\iint_{B_1}h(|\ol{W}(x,y)-W(x,y)|)\dd x\dd y&\leq\Big(a'_1 \frac{\log m}{m^\alpha}+a'_2 \frac{1}{m^\alpha}\Big)\iint_{B_1}\dd x\dd y\\
&=\frac{\ell^{2d}}{m^{2d}}\Big(a'_1 \frac{\log m}{m^\alpha}+a'_2 \frac{1}{m^\alpha}\Big)
\end{align*}
where the constants $a'_1,a'_2$ are independent of the particular cell in question.
For $B_2\in C_2$, $W$ is not H\"{o}lder continuous and so we have the bound
\begin{equation}|\ol{W}(x,y) - W(x,y)|\leq\sup_{(x,y),(x',y')\in B_2}\abs{W(x,y)-W(x',y')}\leq\rho\end{equation}
where
\begin{equation}\rho = \max_{t,t'\in\set{1,\dots,T}}\sup_{\substack{(x,y)\in A_t\\(x',y')\in A_{t'}}}\abs{f_t(x,y) - f_{t'}(x',y')}.\end{equation}
Each $f_t$ takes values in $[0,1]$, so $\rho\in[0,1]$. Then in the same manner as above
\begin{equation}h(|\ol{W}(x,y)-W(x,y)|)\leq -\rho\log\rho +\frac{\rho}{\ln 2}\equiv \theta.\end{equation}
Then
\begin{equation}\iint_{B_2}h(|\ol{W}(x,y)-W(x,y)|)\dd x\dd y\leq\theta\iint_{B_2}\dd x\dd y = \frac{\ell^{2d}}{m^{2d}}\theta.\end{equation}
Putting everything together, we have
\begin{align*}
0\leq\Delta_n\leq\frac{dn\log(m)}{\binom{n}{2}}+&\sum_{B_1\in C_1}\iint_{B_1}h(|\ol{W}(x,y)-W(x,y)|)\dd x\dd y\\ &+\sum_{B_2\in C_2}\iint_{B_2}h(|\ol{W}(x,y)-W(x,y)|)\dd x\dd y,
\end{align*}
and so
\begin{align*}
0\leq\Delta_n&\leq\frac{dn\log m}{\binom n2} +\frac{|C_1|}{m^{2d}}\ell^{2d}\Big(a'_1 \frac{\log m}{m^\alpha}+a'_2 \frac{1}{m^\alpha}\Big)+\frac{|C_2|}{m^{2d}}\ell^{2d}\theta\\
&\leq\frac{dn\log m}{\binom n2} +a_1 \frac{\log m}{m^\alpha}+a_2 \frac{1}{m^\alpha}+\frac{4d(T-1)}{m}\ell^{2d}\theta
\end{align*}
using \cref{c2bnd} and that $|C_2|/m^{2d}\leq 1$. 
If no collection of the cells partitions $\K$, then $O(m^{d-1})$ cells overlap with the boundary of $\K$. We divide the cells into two collections: $C$ consists of cells completely contained within $\K$, and $\partial C$ consists of cells which intersect the boundary $\partial\K$. Define $\ol{W}$ as above, and we trivially extend $W$ and $\ol{W}$ by
\begin{equation}W(x,y)=\ol{W}(x,y) = 0\end{equation}
for  $(x,y)\notin\K^2$. Following the work above, we need to bound the term in \cref{int}. We can express the integral as
\begin{align*}
\iint_{\K^2}h(|\ol{W}(x,y)-W(x,y)|)\dd x\dd y &=\sum_{\substack{i,j\\ c_i,c_j\in C}}\iint_{c_i\times c_j}h(|\ol{W}(x,y)-W(x,y)|)\dd x\dd y\\
&+\sum_{\substack{i,j\\ c_i,c_j\in\partial C}}\iint_{c_i\times c_j}h(|\ol{W}(x,y)-W(x,y)|)\dd x\dd y\\
&+2\sum_{\substack{i,j\\ c_i\in C, c_j\in\partial C}}\iint_{c_i\times c_j}h(|\ol{W}(x,y)-W(x,y)|)\dd x\dd y
\end{align*}
The previous analysis applies to the first term. In the latter two terms, we cannot use the H\"{o}lder continuity of $W$, so we bound the binary entropy function trivially by 1. By the argument from Baker et al. \cite{baker:2026}, we have
\begin{align*}
\sum_{\substack{i,j\\ c_i,c_j\in\partial C}}\iint_{c_i\times c_j}h(|\ol{W}(x,y)-W(x,y)|)\dd x\dd y&\leq\sum_{\substack{i,j\\ c_i,c_j\in\partial C}}\iint_{c_i\times c_j}\dd x\dd y\\
&\leq |\partial C|^2\Big(\frac{\ell}{m}\Big)^2\\
&\leq |\partial\K|\Big(\frac{\ell}{m}\Big)^2,
\end{align*}
and
\begin{equation}
2\sum_{\substack{i,j\\ c_i\in C, c_j\in\partial C}}\iint_{c_i\times c_j}h(|\ol{W}(x,y)-W(x,y)|)\dd x\dd y\leq 2|\partial\K|\frac{\ell}{m}.
\end{equation}
Letting $a_4 = |\partial\K|^2\ell^2$, we have
\[
0\leq\Delta_n\leq\frac{dn\log m}{\binom n2} +a_1 \frac{\log m}{m^\alpha}+a_2 \frac{1}{m^\alpha}+a_3\frac{1}{m}+a_4\frac{1}{m^2}
\]
where $a_3 = 4d(T-1)\ell^{2d}\theta + 2|\partial\K|\ell$.

\end{document}